\documentclass[12pt]{amsart}
\usepackage[latin9]{inputenc}
\usepackage{amstext}
\usepackage{amsthm}
\usepackage{amssymb}
\usepackage[unicode=true,
 bookmarks=false,
 breaklinks=false,pdfborder={0 0 1},backref=section,colorlinks=false]{hyperref}
\usepackage{verbatim}
\usepackage{amstext}
\usepackage{amsthm}
\usepackage{eurosym}
\usepackage{latexsym}
\usepackage{amsthm}
\usepackage{amsfonts}
\usepackage{ulem}

\makeatletter
\numberwithin{equation}{section}
\numberwithin{figure}{section}
\numberwithin{equation}{section}
\numberwithin{figure}{section}
\numberwithin{equation}{section}
\newtheorem{theorem}{Theorem}

\newtheorem{corollary}[theorem]{Corollary}
\newtheorem{proposition}[theorem]{Proposition}

\theoremstyle{definition}

\newtheorem{example}[theorem]{Example}
\newtheorem*{acknowledgements*}{Acknowledgements}
\theoremstyle{remark}

\numberwithin{theorem}{section}   
\subjclass[2010]{Primary 46B42; Secondary 46A45, 46B20}
\keywords{Banach lattice, lower, upper estimate, symmetric sequence space, upper, lower semi-homogeneous basis, Lorentz space, Orlicz space}
\makeatother

\begin{document}
\title[Optimal upper and lower sequence spaces]{Optimal upper and lower sequence spaces with applications}
\author{Sergey V. Astashkin}
\address{Astashkin: Department of Mathematics, Samara National Research
University, Moskovskoye shosse 34, 443086, Samara, Russian Federation;
Department of Mathematics, Bahcesehir University, 34353, Istanbul, Turkey}
\email{astash56@mail.ru}
\author{Per G. Nilsson}
\address{Nilsson: Stockholm, Sweden}
\email{pgn@plntx.com}
\date{\today }

\begin{abstract}
We study the optimal upper $X_U$ and lower $X_L$ sequence spaces that can be assigned to each Banach lattice $X$. These spaces are symmetric,  have the Fatou property and the unit vector basis has in these spaces very special properties. Determined by the order structure of $X$ the spaces $X_U$ and $X_L$ turn out to be very useful when studying Banach lattices. Among other results, in terms of these constructions, we identify Banach lattices that satisfy equal-norm upper and lower $p$-estimates, give a characterization of $L_p(\mu)$-spaces, derive some properties of the tensor product operator in Lorentz and  Orlicz spaces, identify Orlicz spaces in which the unit vector basis is upper semi-homogeneous.
\end{abstract}

\maketitle

\section{Introduction}

\label{Intro}

Recently, in the paper \cite{AN1},  the authors have initiated the study of the  optimal upper/lower sequence spaces  assigned to Banach lattices. The introduction of these concepts was motivated primarily by the study of connections between upper/lower estimates for Banach lattices and the notion of $s$-decomposibility (see \cite{AN1} for the definition), which has roots in the interpolation theory of operators. The use of the spaces $X_U$ and $X_L$ constructed by a Banach lattice $X$ made it possible to obtain in \cite{AN1} a characterization of relatively $s$-decomposable Banach lattices. As an application of the latter result, it was obtained in particular an orbital factorization
of relative $K$-functional estimates for Banach couples $\vec{X}%
=(X_{0},X_{1})$ and $\vec{Y}=(Y_{0},Y_{1})$ through some suitable couples of weighted $L_{{p}}$-spaces provided if $X_{i}$ and $Y_{i}$ are relatively $s$%
-decomposable for $i=0,1$. 

The first aim of this paper is a continuation of the study of properties of the optimal upper/lower sequence spaces $X_U$ and $X_L$ which was began in \cite{AN1}. We show that these spaces are symmetric for each Banach lattice $X$, have the Fatou property and satisfy some duality type embeddings. 

The second line of our research is connected with some applications  of the above constructions to the study of Banach lattices. First, in terms of embeddings of the spaces $X_{U}$ and $X_{L}$,  we identify Banach lattices that satisfy {equal-norm} upper and lower $p$-estimates (see Proposition{\ref{Tsir4}) and give a characterization of $L_p(\mu)$-spaces (see Theorem \ref{Th-equal-optimal-space-class}). Also, in Theorem  \ref{add cor1}, we prove that a symmetric sequence space $E$ coincides with a space $X_{U}$ (resp. $X_{L}$) for some Banach lattice $X$ if and only if the unit vector basis is upper (resp. lower) semi-homogeneous in $E$. 
Recall that a basis $\{x_k\}_{k=1}^\infty$ in a Banach space $X$ is called upper (resp. {\it lower}) {\it semi-homogeneous} if it is bounded and every bounded block basic  sequence $\{y_k\}$ of $\{x_k\}$ is dominated by (resp. dominates) $\{x_k\}$.


In the final part of the paper we apply the obtained results to the study of symmetric Lorentz and Orlicz sequence spaces. Under some conditions we identify upper/lower optimal spaces for them that allows us in particular to recover some  properties of the tensor product operator in these spaces. Among other results, in Propositions \ref{Tsir7} and \ref{Tsir8}, we prove that the unit vector basis is upper (resp. lower) semi-homogeneous in the Orlicz space ${\ell_N}$ if and only if $N(uv)\le CN(u)N(v)$ (resp. $N(u)N(v)\le CN(uv)$) for some constant $C>0$ and all $0<u,v\le 1$. Recall that a similar criterion for the lower semi-homogeneity of the Lorentz sequence spaces was obtained by Casazza and Bor-Luh Lin in \cite{CaLin} (by using a completely different method).

\vskip0.5cm

\section{Preliminaries.}
\label{Sec-definitions}

\subsection{Banach lattices, upper and lower estimates for
disjoint elements.}
\label{Ban latt}


We will assume that the reader is familiar with the definition of Banach
lattices, their basic properties and some basic terminology (see, for
instance, \cite{LT79}, \cite{MeNi91}, \cite{Sch12}). For definiteness, all Banach spaces and lattices considered in this paper are assumed to be real and infinite dimensional.





Let us start with recalling the notions of lower and upper estimates in
Banach lattices; see \cite[Definition 1.f.4]{LT79}. 
A Banach lattice $X$ is
said to \textit{satisfy an upper (resp. a lower) $p$-estimate}, where $p\in %
\left[ 1,\infty \right] $, if for some constant $M<\infty $ and all finite
sequences of pairwise disjoint elements $\left\{ x_{i}\right\}
_{i=1}^{n}\subseteq X$ it holds
\begin{equation}  \label{upper est}
\Big\Vert \sum_{i=1}^{n} x_{i} \Big\Vert _{X}\leq M\Big(
\sum_{i=1}^{n}\Vert x_{i}\Vert _{X}^{p}\Big) ^{1/p}
\end{equation}%
(resp. 
\begin{equation}  \label{left est}
\Big( \sum_{i=1}^{n}\Vert x_{i}\Vert _{X}^{p}\Big) ^{1/p}\leq
M\Big\Vert \sum_{i=1}^{n} x_{i} \Big\Vert\;).
\end{equation}

The infimum of all $M$, satisfying the above inequality, is denoted by $M^{%
\left[ p\right] }\left( X\right) $ and $M_{\left[ p\right] }\left( X\right) $%
, respectively. Note that every Banach lattice admits (trivially) an upper $%
1 $-estimate and a lower $\infty $-estimate.

Recall that the Grobler-Dodds indices $\delta (X)$ and $\sigma (X)$ of a
Banach lattice $X$ are defined by 
\begin{equation*}
\delta (X):=\sup \{p\geq 1:\,X\;\text{satisfies an upper $p$-estimate}\}
\end{equation*}%
and 
\begin{equation*}
\sigma (X):=\inf \{p\geq 1:\,X\;\text{satisfies a lower $p$-estimate}\}.
\end{equation*}%
For every Banach lattice $X$ we have $1\leq \delta
(X)\leq \sigma (X)\leq \infty $. Moreover, the following duality relations
hold: 
\begin{equation*}
\frac{1}{\delta (X)}+\frac{1}{\sigma (X^{\ast })}=1\;\;\mbox{and}\;\;\frac{1%
}{\sigma (X)}+\frac{1}{\delta (X^{\ast })}=1.
\end{equation*}

Further, the following somewhat weaker estimates for disjoint elements in
Banach lattices will be also of interest for us. A Banach lattice $X$ is
said to \textit{satisfy an equal-norm upper (resp. an equal-norm lower) $p$%
-estimate}, where $p\in \left[ 1,\infty \right] $, if estimate 
\eqref{upper est} (resp. \eqref{left est}) holds for some constant $M>0$ and all
finite sequences of pairwise disjoint elements $\left\{ x_{i}\right\}
_{i=1}^{n}\subseteq X$ such that $\|x_i\|_X=1$, $i=1,\dots,n$. In other
words, this means that 
\begin{equation*}
\Big\|\sum_{i=1}^{n} x_{i} \Big\|_{X}\leq Mn^{1/p}\;\;\mbox{(resp.}\;\;M%
\Big\|\sum_{i=1}^{n} x_{i} \Big\|\ge n^{1/p}).
\end{equation*}

Clearly, every Banach lattice that admits an upper (resp. a lower) $p$%
-estimate admits an equal-norm upper (resp. an equal-norm lower) $p$%
-estimate as well. Moreover, it is known (see \cite[Lemma 2.3]{CwNiSc03} or \cite[Theorem 2]{Byl}) that 
\begin{equation*}
\delta(X)=\sup \{p\geq 1:\,X\;\text{satisfies an equal-norm upper $p$%
-estimate}\}
\end{equation*}%
and 
\begin{equation*}
\sigma (X)=\inf \{p\geq 1:\,X\;\text{satisfies an equal-norm lower $p$%
-estimate}\}.
\end{equation*}
At the same time, there are Banach lattices which admit an equal-norm upper
(resp. an equal-norm lower) $p$-estimate, but fail to admit an upper (resp.
a lower) $p$-estimate. We refer, as an example, to the space $V_{\gamma,p}$, introduced by Tzafriri in \cite{Tz}.







Let $1\le p\le\infty$. We say that ${\ell_p}$ is \textit{finitely
lattice representable} in a Banach lattice $X$ whenever for every $n\in%
\mathbb{N}$ and each $\varepsilon >0$ there exist pairwise disjoint elements $%
x_{i}\in X$, $i=1,2,\dots,n$, such that for any $a_{i}\in \mathbb{R}$, $i=1,\dots,n$  we have

\begin{equation}  \label{Fin distr}
\left( \sum_{i=1}^{n}\left\vert a_{i}\right\vert ^{p}\right) ^{1/p}\leq
\left\Vert \sum_{i=1}^{n}a_{i}x_{i}\right\Vert _{X}\leq \left( 1+\varepsilon
\right) \left( \sum_{i=1}^{n}\left\vert a_{i}\right\vert ^{p}\right) ^{1/p}.
\end{equation}

Recall that a Banach lattice $X$ is \textit{order continuous} if, for every downward directed set $\{x_\alpha\}_{\alpha\in A}\subset X$ with $\bigvee_{\alpha\in A}x_\alpha=0$, we have $\lim_{\alpha}\|x_\alpha\|_X=0$.


\vskip0.3cm

\def\hj{{\mathbb R}}
\def\hk{{\mathbb Z}}
\def\fg{{\mathbb N}}

\subsection{Banach function spaces}
\label{Function lattices}

Let $(T,\Sigma,\mu)$ be a complete $\sigma$-finite measure space. The set of all a.e. finite real-valued functions (of equivalence classes) defined on $T$ with natural algebraic operations and the topology of convergence in measure $\mu$ on sets of finite measure is a linear topological space. We denote it by $S(T,\Sigma,\mu)$. The notation $x \leq y$, where $x,y \in S(T,\Sigma,\mu)$, means that there is a set $A\subset T$ such that $\mu(T\setminus A)=0$ and $x(t) \leq y(t)$ for $t\in A$ (we say that $x \leq y$ a.e. on $T$). 

A linear subspace $X \subset S(T,\Sigma,\mu)$ is said to be a {\it function lattice}  if $\{x \in S(T,\Sigma,\mu):\,|x| \leq |y|\}\subset X$ for every $y \in X$. Any function lattice $X$, equipped with a norm such that $\|x\| \leq \|y\|$ whenever $x,y \in X$ and $|x| \leq |y|$, is called a {\it normed function space}. If $X$ is complete as a normed linear space, $X$ is said to be a {\it Banach function space (or lattice)}.


Any Banach function space $X$ is linearly and continuously embedded into the corresponding linear topological space $S(T,\Sigma,\mu)$ (see e.g., \cite[Theorem~4.3.1]{KA}). This means that from the convergence in norm it follows the convergence in measure on all subsets of $T$ of finite measure.

If $X$ is a Banach function space, then the {\it K\"{o}the dual} 
(or {\it associated}) function lattice $X'$ consists of all $y\in S(T,\Sigma,\mu)$ such that
$$
\|y\|_{X'}:=\sup\,\biggl\{\int_{T}{x(t)y(t)\,d{\mu}}:\;\;
\|x\|_{X}\,\leq{1}\biggr\}\,<\,\infty.
$$
Any K\"{o}the dual space $X'$ is complete with respect to the norm $y\mapsto \|y\|_{X'}$ and is embedded isometrically into (Banach) dual space $X^*.$ Moreover, $X$ is continuously embedded into its second K\"{o}the dual $X''$, and $\|x\|_{X''}\le\|x\|_X$ for $x\in X$. A Banach function space $X$ is embedded into $X''$  isometrically if and only if the norm of $X$ is {\it order semi-continuous}, i.e.,  if the conditions $x_n\in X,$ $n=1,2,\dots,$
$x\in{X}$, and $x_n\to{x}$ a.e. on $T$ imply that $||x||_X\le
\liminf_{n\to\infty}{||x_n||_X}$ \cite[Theorem~6.1.6]{KA}. A Banach function space $X$ {\it has the Fatou property} (or {\it maximal}) if from $x_n\in
X,$ $n=1,2,\dots,$ $\sup_{n=1,2,\dots}\|x_n\|_X<\infty$, $x\in
S(T,\Sigma,\mu)$ and $x_n\to{x}$ a.e. on $T$ it follows that $x\in X$ and $||x||_X\le \liminf_{n\to\infty}{||x_n||_X}.$ Observe that a Banach function lattice $X$ has the Fatou property if and only if the natural inclusion of $X$ into $X''$ is a surjective isometry \cite[Theorem~6.1.7]{KA}. 



In the case when the measure space $(T,\Sigma,\mu)$ is the set of positive integers  with the counting measure $\mu$ defined by $\mu(\{k\})=1,$ $k \in {\fg}$, we will say about {\it Banach sequence spaces (or lattices)}. If $E$ is such a space, the {\it K\"{o}the dual}  sequence lattice $E'$ is the set of all sequences $y=(y_k)_{k=1}^\infty$ such that
$$
\|y\|_{E'}:=\sup\,\biggl\{\sum_{k=1}^\infty x_ky_k:\;\;
\|x\|_{E}\,\leq{1}\biggr\}<\infty.
$$

For a detailed exposition of the theory of Banach function and sequence spaces see the monographs \cite{BSh}, \cite{KA}, \cite{LT79}.

\vskip 0.3cm


\subsection{Symmetric sequence spaces}
\label{RI}




Let $(a_{k})_{k=1}^{\infty }$ be a bounded sequence of real numbers. In what follows, $(a_{k}^{\ast })_{k=1}^{\infty }$ denotes the nonincreasing
permutation of the sequence $(|a_{k}|)_{k=1}^{\infty }$ defined by 
\begin{equation*}
a_{k}^{\ast }:=\inf_{\mathrm{card}\,A=k-1}\sup_{i\in\mathbb{N}\setminus
A}|a_i|,\;\;k\in\mathbb{N}.
\end{equation*}

A Banach sequence space $E$ is called \textit{symmetric} if 
$E\subset {\ell_\infty}$ and the conditions $y_{k}^{\ast }=x_{k}^{\ast }$, $%
k=1,2,\dots $, $x=(x_{k})_{k=1}^{\infty }\in E$ imply that $%
y=(y_{k})_{k=1}^{\infty }\in E$ and $\Vert y\Vert _{E}=\Vert x\Vert _{E}$.

For every symmetric sequence space $E$ the following one norm continuous embeddings hold:
${\ell_1}\overset{1}{\hookrightarrow }E\overset{1}{\hookrightarrow }{\ell_\infty}$ \cite[Theorem II.4.1]{KPS}. 
Moreover, if $a=(a_k)_{k=1}^\infty\in E$ and $\pi$ is an arbitrary permutation of $\mathbb{N}$, then the sequence $a_\pi=(a_{\pi(k)})_{k=1}^\infty$ belongs to $E$ and $\|a_\pi\|_E=\|a\|_E.$

Let $E$ be a symmetric sequence space. The \textit{fundamental function} of
$E $ is defined by the formula 
\begin{equation*}
\phi_{E}(n):=\Big\|\sum_{k=1}^n e_k\Big\|_{E},\;\;n=1,2,\dots.
\end{equation*}
In what follows, $e_k$ are the canonical unit vectors, i.e., $%
e_k=(e_k^i)_{i=1}^\infty$, $e_k^i=0$ for $i\ne k$ and $e_k^k=1$, $%
k,i=1,2,\dots$. 
For every symmetric space $E$, $\phi_{E}$ is a nondecreasing and positive function such that $\varphi _E(n)/n$ is nonincreasing.

Let $\phi_E=\phi_E(n)$, $n\in\mathbb{N}$, be the fundamental function of a
symmetric sequence space $E$. We introduce the following \textit{%
dilation functions} of $\phi_E$: 
\begin{equation*}
M_E^0(n):=\sup_{m\in\mathbb{N}}\frac{\phi_E(m)}{\phi_E(mn)}\;\;\mbox{and}%
\;\;M_E^\infty(n):=\sup_{m\in\mathbb{N}}\frac{\phi_E(mn)}{\phi_E(m)},\;\;n\in%
\mathbb{N},
\end{equation*}
and the \textit{fundamental indices} of $E$ by 
\begin{equation}
\label{gen ind}
\mu_E= -\lim_{n\to\infty}\frac{1}{n}\log_2 M_E^0(2^{n})\;\;\mbox{and}%
\;\;\nu_E= \lim_{n\to\infty}\frac{1}{n}\log_2 M_E^\infty(2^{n})
\end{equation}
\cite[Chapter 3, Exercise~14, p.~178]{BSh}. We have always that $0\le \mu_E\le \nu_E\le 1$ (see \cite[\S\,II.4.4]{KPS}).

The most important examples of symmetric sequence spaces are the ${\ell_p}$%
-spaces, $1\le p\le\infty,$ with the usual norms 
\begin{equation*}
{\|a\|}_{{\ell_p}}:=\left\{ 
\begin{array}{ll}
{\left( {\sum}_{k=1}^{\infty}{|a_k|}^{p} \right)}^{1/p}\;,\; & 1 \leq p <
\infty \\ 
\sup\limits_{k=1,2,\dots}|a_k|\;,\; & p=\infty%
\end{array}%
. \right.
\end{equation*}

Their generalization, the ${\ell_{p,q}}$-spaces, $1<p<\infty,$ $1\le
q\le\infty,$ are equipped with the quasi-norms 
\begin{equation*}
\|a\|_{p,q}:=\Big(\sum_{k=1}^\infty (a_k^*)^q(k^{q/p}-(k-1)^{q/p})\Big)%
^{1/q}<\infty\;\;\mbox{for}\;\;q<\infty,
\end{equation*}
and 
\begin{equation*}
\|a\|_{p,\infty}:=\sup_{n\in\mathbb{N}}n^{1/{p-1}}\sum_{k=1}^n a_k^*<\infty
\end{equation*}
(see e.g. \cite{D01}). 
The functional $a\mapsto \|a\|_{{\ell_{p,q}}}$ does not
satisfy the triangle inequality for $p<q\le\infty$, but it is equivalent to
some symmetric norm (see, e.g., \cite[Theorem~4.4.3]{BSh}). Clearly, $\ell_{p,p}={\ell_p}$, $1<p<\infty,$ isometrically.
Moreover, for every $1<p<\infty,$ $1\le q\le\infty,$ we have $({\ell_{p,q}})^{\prime }=\ell_{p',q'}$, where $1/p+1/p'=1$, $1/q+1/q'=1$.

If the fundamental function $\phi_E$ of a symmetric sequence space $E$ satisfies
the estimate $\phi_E(m)\le Cm^{1/p}$ (resp. $m^{1/p}\le C\phi_E(m)$) for
some $C>0$ and all $m\in\mathbb{N}$, we have that ${\ell_{p,1}}\overset{C}{\hookrightarrow } E$ (resp. $E\overset{C}{\hookrightarrow } {\ell_{p,\infty}}$) (see \cite[Lemma II.5.2]{KPS}).

In turn, the $\ell_{p,q}$-spaces (when $1\leq q\leq p<\infty $) belong to a wider
class of the Lorentz spaces. Let $1\leq q<\infty $, and let $%
\{w_{k}\}_{k=1}^{\infty }$ be a nonincreasing sequence of positive numbers such that $w_1=1$, $\lim_{k\to\infty} w_k=0$ and $\sum_{k=1}^\infty w_k=\infty$.
The Lorentz space $\lambda _{q}(w)$ (see, e.g., \cite{Lo-51} or \cite[%
Chapter~4e]{LT77}) consists of all sequences $a=(a_{k})_{k=1}^{\infty }$
satisfying 
\begin{equation*}
\Vert a\Vert _{\lambda _{q}(w)}:=\Big(\sum_{k=1}^{\infty }(a_{k}^{\ast
})^{q}w_{k}\Big)^{1/q}<\infty .
\end{equation*}%
Since the classical Hardy-Littlewood inequality (see, e.g., \cite[%
Theorem~2.2.2]{BSh}) yields 
\begin{equation*}
\Vert a\Vert _{\lambda _{q}(w)}=\sup_{\pi }\Big(\sum_{k=1}^{\infty }|a_{\pi
(k)}|^{q}w_{k}\Big)^{1/q},
\end{equation*}%
where the supremum is taken over all permutations $\pi $ of the set of
positive integers, the functional $a\mapsto \Vert a\Vert _{\lambda _{q}(w)}$
defines on the space $\lambda _{q}(w)$ a symmetric norm. Every Lorentz
sequence space is separable and has the Fatou property.

Clearly, the fundamental function of $\lambda_q(w)$ is defined by 
\begin{equation}  \label{fund Lor}
\phi_{\lambda_q(w)}(n)=\Big(\sum_{k=1}^n w_k\Big)^{1/q},\;\;n\in\mathbb{N}.
\end{equation}
Therefore, by \eqref{gen ind}, the fundamental indices of $\lambda_q(w)$ can be calculated by the
formulae 
\begin{equation}  \label{Boyd ind Lor}
\mu_{\lambda_q(w)}=-\lim_{n\to\infty}\frac{1}{n}\log_2\sup_{j\in\mathbb{N}%
}\left(\frac{\sum_{k=1}^{j} w_k}{\sum_{k=1}^{2^nj} w_k}\right)^{1/q}.
\end{equation}
and 
\begin{equation}  \label{Boyd ind Lor1}
\nu_{\lambda_q(w)}= \lim_{n\to\infty}\frac{1}{n}\log_2\sup_{j\in\mathbb{N}%
}\left(\frac{\sum_{k=1}^{2^nj} w_k}{\sum_{k=1}^{j} w_k}\right)^{1/q}.
\end{equation}

Another natural generalization of the $\ell _{p}$-spaces are the Orlicz spaces
(see \cite{KR61}, \cite{Mal}, \cite{Mus}, \cite{RR}). Let $N$ be
an Orlicz function, that is, an increasing convex continuous function on $%
[0,\infty )$ such that $N(0)=0$ and $\lim_{t\rightarrow \infty }N(t)=\infty $%
. The \textit{Orlicz sequence space} $l_{N}$ consists of all sequences $%
a=(a_{k})_{k=1}^{\infty }$ such that 
\begin{equation*}
\Vert a\Vert _{l_{N}}:=\inf \left\{ u>0:\,\sum_{k=1}^{\infty }N\Big(\frac{%
|a_{k}|}{u}\Big)\leq 1\right\} <\infty .
\end{equation*}%
Without loss of generality, we will assume that $N(1)=1$. In particular, if $%
N(s)=s^{p}$, $1\leq p<\infty $, we obtain $\ell ^{p}$ with the usual norm.

Observe that the definition of the space $\ell_{N}$ depends (up to equivalence of norms) only on the behaviour of the function $N$ near zero.

Every Orlicz space ${\ell_N}$ has the Fatou property and it is separable if and
only if the function $N$ satisfies the \textit{$\Delta_2$-condition at zero}%
, i.e., 
\begin{equation*}
\limsup_{u\to 0}\frac{N(2u)}{N(u)}<\infty.
\end{equation*}
Since the fundamental function $\phi_{{\ell_N}}$ is defined by 
\begin{equation}  \label{fund Orl}
\phi_{{\ell_N}}(n)=\frac{1}{N^{-1}(1/n)},\;\;n\in\mathbb{N},
\end{equation}%
where $N^{-1}$ is the inverse function for $N$, one can readily check that
the fundamental indices of ${\ell_N}$ can be calculated by the formulae: 
\begin{equation}  \label{fund ind Orl}
\mu_{{\ell_N}}=\sup\Big\{1/q:\,\inf_{0<s,t\le 1}\frac{N(st)}{N(s)t^q%
}>0\Big\}\;\;\mbox{and}\;\;\nu_{{\ell_N}}=\inf\Big\{1/q:\,\sup_{0<s,t\le 1}\frac{N(st)}{N(s)t^q}<\infty\Big\}
\end{equation}
(similar formulae for Orlicz spaces on $[0,1]$ can be found in \cite[Proposition 2.b.5]{LT79}).

Finally, $\ell_{N}^{\prime }=\ell_{\tilde{N}} $, with the
Young conjugate function $\tilde{N}$.



\vskip0.3cm

\subsection{Lower/upper semi-homogeneous bases in Banach spaces}
\label{Lower/upper semi-homogeneous}

A basis $\{x_k\}_{k=1}^\infty$ in a Banach space $X$ is called {\it lower} (resp. {\it upper}) {\it semi-homogeneous} if it is bounded and every bounded block basic  sequence $\{y_k\}$ of $\{x_k\}$ dominates (resp. is dominated by) $\{x_k\}$. In particular, if a basis is both lower and upper semi-homogeneous, then it is  perfectly homogeneous and so it is equivalent, by a result of Zippin \cite{Zip}, to the unit vector basis  of $c_0$ or ${\ell_p}$, $1\le p<\infty$. 

Let $E$ be a symmetric sequence space. Then, if the unit vector basis $\{e_k\}_{k=1}^\infty$ is lower (resp. upper) semi-homogeneous in $E$, we have, for any $m\in\mathbb{N}$ and every pairwise disjoint elements $u_{k}\in E$, $k=1,2,\dots,m$, that 
\begin{equation*}
\Big\Vert \sum_{k=1}^{m}u_{k}\Big\Vert _{E}\leq \Big\Vert
\sum_{k=1}^{m}\Vert u_{k}\Vert _{E}e_{k}\Big\Vert _{E}
\end{equation*}
(resp. 
\begin{equation*}
\Big\Vert \sum_{k=1}^{m}\Vert u_{k}\Vert
_{E}e_{k}\Big\Vert _{E}\leq \Big\Vert
\sum_{k=1}^{m}u_{k}\Big\Vert _{E}\;).
\end{equation*}

In what follows, we call a Banach sequence lattice $E$ {\it lower (resp. upper) semi-homogeneous} if the unit vector basis $\{e_k\}_{k=1}^\infty$ is lower (resp. upper) semi-homogeneous in $E$.

Let $F_1$ and $F_2$ be two positive functions (quasinorms). We write $F_1\asymp F_2$ if there exists a positive constant $C$ that does not depend
on the arguments of $F_1$ and $F_2$ such that $C^{-1}F_1\leq F_2\le CF_1$.
Finally, for any finite set $E\subset\mathbb{N}$ by $|E|$ we denote cardinality
of $E$. 

\vskip0.5cm

\section{\label{Sec-optimal-spaces}Optimal Upper and Lower Sequence Lattices}

\subsection{\label{subsec-optimal-def}Definitions and general properties}

We start with recalling the definition and some general properties of a special kind of sequence spaces generated by some appropriate sequences of norms on $\mathbb{R}^n$, $n\in\mathbb{N}$ (see \cite{AN1}).

Let $X$ be an (infinite dimensional) Banach lattice, $S_{X}:=\{x\in X:\,\Vert x\Vert _{X}=1\}$. For each positive integer $n,$ let $\mathfrak{B}_{n}\left( X\right) $ denote the set
of all sequences $\left\{ x_{i}\right\} _{i=1}^{n}\subseteq S_{X}$ of pairwise disjoint elements. Observe that these sets are non-empty for every $n$ (see Lemma 3.2 in \cite{AN1}). Based on $X$, we construct the \textit{upper} $X_{U} $ and the  \textit{lower} $X_{L}$ \textit{optimal sequence spaces} that satisfy the following norm one continuous
embeddings: 
\begin{equation}
{\ell_1}\overset{1}{\hookrightarrow }X_{U}\overset{1}{\hookrightarrow }X_{L}%
\overset{1}{\hookrightarrow }{\ell_\infty}. 
\label{embeddings}
\end{equation}%

To construct $X_{U}$ we define first, for each fixed positive integer $n$, the
following norm on $\mathbb{R}^{n}$ by 
\begin{equation*}
\left\Vert (a_{i})_{i=1}^{n}\right\Vert _{X_{U}\left(
n\right) }:=\sup \left\{ \left\Vert \sum_{i=1}^{n}a_{i}x_{i}\right\Vert
_{X}:\left\{ x_{i}\right\} _{i=1}^{n}\in \mathfrak{B}_{n}\left( X\right)
\right\} .
\end{equation*}%
Let $X_{U}$ then be the space of all real-valued sequences $a=(a_{i})_{i=1}^{\infty }$, for which%
\begin{equation*}
\left\Vert a\right\Vert _{X_{U}}:=\sup_{n}\left\Vert (a_{i})_{i=1}^{n}\right\Vert _{X_{U}\left( n\right) }<\infty .
\end{equation*}%
Since 
\begin{equation*}
\left\Vert (a_{i})_{i=1}^{n}\right\Vert _{X_{U}\left(
n\right) }\leq \left\Vert (a_{i})_{i=1}^{n}\right\Vert _{\ell_1^n}:= \sum_{i=1}^{n}\left\vert a_{i}\right\vert ,\;\;n\in \mathbb{N},
\end{equation*}%
we get the left-hand side embedding in $\left( \ref{embeddings}\right) $. Clearly, if $X$ satisfies an
upper $p$-estimate, we have 
\begin{equation*}
\left\Vert a\right\Vert _{X_{U}}\leq M^{\left[ p\right] }\left( X\right)
\left\Vert a\right\Vert _{{\ell_p}}.
\end{equation*}

The first step in the definition of the space $X_{L}$ is the introduction of
the functional $\Phi$, defined for $a=(a_{i})_{i=1}^{n}\in 
\mathbb{R}^{n}$ by 
\begin{equation*}
\Phi _{n}\left( a\right) :=\inf \left\{ \left\Vert
\sum_{i=1}^{n}a_{i}x_{i}\right\Vert _{X}:\,\left\{ x_{i}\right\}
_{i=1}^{n}\in \mathfrak{B}_{n}\left( X\right) \right\} .
\end{equation*}%
Next, we set
\begin{equation*}
\left\Vert a\right\Vert _{X_{L}\left( n\right) }:=\inf \left\{ \sum_{k\in
F}\Phi _{n}\left( a^{k}\right) :\,F\subseteq \mathbb{N},\left\vert
F\right\vert <\infty ,a^{k}\in \mathbb{R}^{n},a=\sum_{k\in F}a^{k}\right\} .
\end{equation*}%
Since
$$
\left\Vert a\right\Vert _{\ell_{\infty
}^{n}}:=\max_{1\leq i\leq n}\left\vert a_{i}\right\vert \leq \Phi
_{n}\left( a\right),$$ 
the functional $a\mapsto \left\Vert a\right\Vert _{X_{L}\left( n\right)
}$ is a norm on $\mathbb{R}^{n}$. Finally, we define $X_{L}$ to be the space
of all real-valued sequences $a=(a_{i})_{i=1}^{\infty }$, for
which 
\begin{equation*}
\left\Vert a\right\Vert _{X_{L}}:=\sup_{n}\left\Vert (a_{i})_{i=1}^{n}\right\Vert _{X_{L}\left( n\right) }<\infty .
\end{equation*}%
One can easily see that the second and third embeddings in $\left( \ref%
{embeddings}\right) $ hold and moreover
\begin{equation*}
\left\Vert a\right\Vert _{{\ell_p}}\leq M_{\left[ p\right] }\left( X\right)
\left\Vert x\right\Vert _{X_{L}},
\end{equation*}%
whenever the Banach lattice $X$ satisfies a lower $p$-estimate.

Note that the construction, which leads to the space $X_{L}$, is close to the one
developed by Junge in \cite{Jun02}.

Next, we proceed with proving some properties of the spaces $X_{U}$ and $X_{L}$, which complement the results obtained in the paper \cite{AN1}. The first result can be treated as a sharp version of Theorem 3.3 from \cite{AN1}.

\begin{theorem}
\label{Th_XL_XU_Prop} For every Banach lattice $X$, the spaces $X_{U}$ and $X_{L}$ are symmetric.
\end{theorem}
\begin{proof}
Observe that this result for the space $X_L$ was obtained in \cite{AN1} in full, while for $X_{U}$ it was proved under the additional condition that ${\ell_\infty}$ is not finitely representable in $X$ (see Section \ref{Ban latt}). Thus, it remains to show that this assumption is in fact superfluous. 

Indeed, from an inspection of the proof of
Theorem 3.3 in \cite{AN1} it follows that the above assumption is used only in
the proof of Lemma 6.2 in the same paper. Specifically, we have to select,
for every positive integer $m$, pairwise disjoint elements $%
w_{1}^{(m)},\dots ,w_{m}^{(m)}\in X$ such that $\Vert
\sum_{k=1}^{m}w_{k}^{(m)}\Vert _{X}=1$ and 
$$
\lim_{m\rightarrow \infty }\min_{j=1,\dots ,m}\Vert w_{j}^{(m)}\Vert _{X}=0.$$ Clearly, the latter
holds if we have the following: 
\begin{equation}
\lim_{n\rightarrow \infty }\sup \Big\{\Big\|\sum_{k=1}^{n}x_{k}\Big\|%
_{X}:\,\left\{ x_{k}\right\} _{k=1}^{n}\in \mathfrak{B}%
_{n}\left( X\right)\Big\}=\infty.  
\label{Ts2a}
\end{equation}%
Assuming that \eqref{Ts2a} is not the case, for some $C>0$
and all $n\in \mathbb{N}$, every $\left\{ x_{k}\right\} _{k=1}^{n}\in \mathfrak{B}%
_{n}\left( X\right) $ and $a_{k}\in \mathbb{R}$, $k=1,\dots ,n$, we get 
\begin{equation*}
\Big\|\sum_{k=1}^{n}a_{k}x_{k}\Big\|_{X}\leq C\max_{k=1,\dots
,n}|a_{k}|=C\Vert (a_{k})_{k=1}^{n}\Vert _{\ell_\infty }.
\end{equation*}%
Hence, $X_{U}={\ell_\infty}$, and the result, we wished, follows  again.
\end{proof}

Denote by $X_{L}^{0}$ (resp. $X_{U}^{0}$) the closed linear span of
all finitely supported sequences in $X_{L}$ (resp. $X_{U}$). 

\begin{corollary}
If $X$ is an arbitrary Banach lattice, then the canonical unit vectors $%
e_{k}$, $k=1,2,\dots $, form a symmetric normalized basis in both spaces $%
X_{U}^{0}$ and $X_{L}^{0}$. 

In particular, if the support of a sequence $a=(a_k)_{k=1}^\infty$ is finite, say, $\{k\in\mathbb{N}:\,a_k\ne 0\}\subset \{1,2,\dots ,n\}$ for some $n\in \mathbb{N}$, we have that $\Vert a\Vert _{X_{L}}=\Vert a\Vert _{X_{L}(n)}$ and $\Vert a\Vert _{X_{U}}=\Vert
a\Vert _{X_{U}(n)}$. 
\end{corollary}


The spaces $X_{U}$ and $X_{L}$ do not need to be order continuous even if $X$ possesses this property (for instance, $(c_0)_U=(c_0)_L={\ell_\infty}$; see \cite[Example 3.5]{AN1}). At the same
time, the construction of the optimal upper and lower sequence spaces
ensures that they \textit{always} have the Fatou property. 

\begin{proposition}
\label{Prop-XL-XU-Fatou}
For every Banach lattice $X$, the spaces $X_{U}$ and $X_{L}$ have the Fatou property. 
\end{proposition}

\begin{proof}

We start with proving this result for the optimal upper space $X_U$. Let $%
a^n=(a_k^n)_{k=1}^\infty\in X_U$, $a_k^n\ge 0$, $\sup_n\|a^n\|_{X_U}<\infty$
and $a_k^n\uparrow a_k$ as $n\to\infty$ for each $k\in\mathbb{N}$. We have
to prove that $a=(a_k)_{k=1}^\infty\in X_U$ and $\|a\|_{X_U}\le
\sup_n\|a^n\|_{X_U}$. 

Let $m\in \mathbb{N}$ and $\varepsilon >0$ be fixed. From the
assumption it follows that for all sufficiently large $n$ we have
\begin{equation}
0\leq a_{k}-a_{k}^{n}\leq \varepsilon ,\;\;k=1,2,\dots ,m.  \label{add0b}
\end{equation}%
Hence, for every $\left\{u_{k}\right\} _{k=1}^{m}\in \mathfrak{B}_{m}\left( X\right) $ we obtain 
\begin{equation*}
\Big\|\sum_{k=1}^{m}a_{k}u_{k}\Big\|_{X}\leq \Big\|\sum_{k=1}^m (a_k-a_{k}^{n})u_k\Big\|_X+\Big\|\sum_{k=1}^m a_{k}^{n}u_k\Big\|_X\le \varepsilon m+\Vert
(a_{k}^{n})_{k=1}^{m}\Vert _{X_{U}(m)},
\end{equation*}%
whence 
\begin{equation}
\Vert (a_{k})_{k=1}^{m}\Vert _{X_{U}(m)}\leq \varepsilon m+\Vert a^{n}\Vert
_{X_{U}}  \label{add0}
\end{equation}%
if $n$ is sufficiently large. Therefore, since $\varepsilon >0$ is arbitrary,
we get 
\begin{equation*}
\Vert (a_{k})_{k=1}^{m}\Vert _{X_{U}(m)}\leq \sup_{n}\Vert a^{n}\Vert
_{X_{U}},
\end{equation*}%
which implies that 
\begin{equation*}
\Vert a\Vert _{X_{U}}\leq \sup_{n}\Vert a^{n}\Vert _{X_{U}}<\infty .
\end{equation*}%
Thus, $a\in X_{U}$. 

Now, we proceed with proving the same property for $X_{L}$. Again, assume that $a^{n}=(a_{k}^{n})_{k=1}^{\infty
}\in X_{L}$, $a_{k}^{n}\geq 0$, $\sup_{n}\Vert a^{n}\Vert _{X_{L}}<\infty $
and $a_{k}^{n}\uparrow a_{k}$ as $n\rightarrow \infty $ for each $k\in 
\mathbb{N}$. 

Let $\varepsilon >0$ be arbitrary. For a fixed $m\in \mathbb{N}$ we find a positive integer $n_{0}$ such that inequality $\left( \ref%
{add0b}\right) $ holds for all $n\geq n_{0}$. Let $n\geq n_{0}$ be fixed.
There is a representation $(a_{k}^{n})_{k=1}^{m}=\sum_{i\in A}b^{i}$, $%
b^{i}=(b_{k}^{i})_{k=1}^{m}$, $|A|<\infty $, such that 
\begin{equation}
\Vert (a_{k}^{n})_{k=1}^{m}\Vert _{X_{L}(m)}\geq \sum_{i\in A}\Phi
_{m}(b^{i})-\varepsilon .  \label{add0c}
\end{equation}%
Setting $d_{k}^{i}:=b_{k}^{i}+\frac{1}{|A|}(a_{k}-a_{k}^{n})$, $i\in A$, $%
k=1,\dots ,m$, we obtain 
\begin{equation*}
\sum_{i\in A}d_{k}^{i}=\sum_{i\in
A}b_{k}^{i}+a_{k} -a_{k}^{n}=a_{k}^{n}+a_{k}-a_{k}^{n}=a_{k},\;\;k=1,\dots ,m,
\end{equation*}%
that is, $(a_{k})_{k=1}^{m}=\sum_{i\in A}d^{i}$, where $%
d^{i}=(d_{k}^{i})_{k=1}^{m}$. Furthermore, in view of $\left( \ref{add0b}%
\right) $, we have 
\begin{equation*}
|d_{k}^{i}-b_{k}^{i}|=\frac{1}{|A|}(a_{k}-a_{k}^{n})\leq
\frac{\varepsilon}{|A|},\;\;i\in A,\;k=1,\dots ,m.
\end{equation*}%
Consequently, as above, for every $\left\{ u_{k}\right\} _{k=1}^{m}\in \mathfrak{B}%
_{m}\left( X\right) $ and each $i\in A$, we get 
\begin{equation*}
\Big\|\sum_{k=1}^{m}d_{k}^{i}u_{k}\Big\|_{X}\leq \frac{m\varepsilon}{|A|}+\Big\|%
\sum_{k=1}^{m}b_{k}^{i}u_{k}\Big\|_{X}.
\end{equation*}%
By the definition of the functional $\Phi _{m}(\cdot )$, this
implies that 
\begin{equation*}
\Phi _{m}(d^{i})\leq \frac{m\varepsilon}{|A|}+\Phi _{m}(b^{i}),\;\;i\in A,
\end{equation*}%
and hence from inequality $\left( \ref{add0c}\right) $ it follows that 
\begin{equation*}
\Vert (a_{k})_{k=1}^{m}\Vert _{X_{L}(m)}\leq \sum_{i\in A}\Phi
_{m}(d^{i})+\varepsilon\leq \sum_{i\in A}\Phi _{m}(b^{i})+\varepsilon (m+1)\leq \Vert
(a_{k}^{n})_{k=1}^{m}\Vert _{X_{L}(m)}+\varepsilon (m+2).
\end{equation*}%
Thus, since $\varepsilon >0$ is arbitrary, for every $m\in \mathbb{N}$ there exists $n_{0}$ such that for all $n\geq n_{0}$ 
\begin{equation*}
\Vert (a_{k})_{k=1}^{m}\Vert _{X_{L}(m)}\leq \Vert
a^{n}\Vert _{X_{L}}.
\end{equation*}%
Hence,
\begin{equation*}
\Vert a\Vert _{X_{L}}=\sup_{m}\Vert (a_{k})_{k=1}^{m}\Vert _{X_{L}(m)}\leq
\sup_{n}\Vert a^{n}\Vert _{X_{L}}<\infty ,
\end{equation*}%
that is, $a\in X_{L}$, and the proof is completed.
\end{proof}

In view of embeddings \eqref{embeddings}, $X_{U}$
and $X_{L}$ are both normalized intermediate spaces with respect to the
Banach couple $\left( {\ell_1},\ell_{\infty }\right).$ Moreover, from Proposition 
\ref{Prop-XL-XU-Fatou} and Calder\'{o}n-Mityagin theorem (see \cite[Theorem 3.2.12]{BSh} or  \cite[Theorem II.4.9]{KPS}) we obtain the following. 

\begin{corollary}
\label{Cor-XL-XU-Interpolation}
For every Banach lattice $X$, the optimal sequence lattices $X_{U}$ and $X_{L}$ are exact interpolation spaces with respect to the couple $\left( {\ell_1},\ell_{\infty }\right) .$ 
\end{corollary}



\vskip0.3cm

\subsection{\label{subsec-duality}{Duality type properties of optimal sequence spaces}}

\begin{proposition}
\label{Prop-Dual.I} Let $X$ be an order semi-continuous Banach function space on some measure space $(T,\Sigma,\mu)$. The following embeddings hold: 
\begin{eqnarray*}
(X')_{L}\overset{1}{\hookrightarrow }(X_{U})'\;\;\mbox{and}\;\;(X')_{U}\overset{1}{\hookrightarrow }\left( X_{L}\right)'.
\end{eqnarray*}
\end{proposition}

\begin{proof}
{To prove the first embedding, we need to show that for every $%
b=(b_{k})\in (X^{^{\prime }})_{L}$ the functional 
\begin{equation*}
\langle a,b\rangle :=\sum_{k=1}^{\infty }a_{k}b_{k},\;\;a=(a_{k})\in X_{U},
\end{equation*}%
is bounded on the space $X_{U}$ and $\Vert b\Vert _{(X_{U})^{\prime }}\leq
\Vert b\Vert _{(X^{^{\prime }})_{L}}$. One can easily see that it suffices
to check that for each $\varepsilon >0$ and all $a=(a_{k})\in X_{U}$, $%
b=(b_{k})\in (X^{^{\prime }})_{L}$ such that $a_{k}\geq 0$, $b_{k}\geq 0$
the following inequality holds: 
\begin{equation}
\sum_{k=1}^{\infty }a_{k}b_{k}\leq (1+\varepsilon )\Vert b\Vert _{(X^{^{\prime
}})_{L}}\Vert a\Vert _{X_{U}}.  
\label{add3a}
\end{equation}
}

{To this end, select $\delta >0$ such that $\frac{1+\delta }{1-\delta 
}<1+\varepsilon $. For any $n\in \mathbb{N}$, by the definition of the functional $\Phi
_{n}^{X'}(\cdot )$, we choose $\left\{ v_{k}\right\}
_{k=1}^{n}\in \mathfrak{B}_{n}\left( X^{^{\prime }}\right) $, $v_k\ge 0$, such that 
\begin{equation}
\Big\|\sum_{k=1}^{n}b_{k}v_{k}\Big\|_{X^{\prime }}\leq (1+\delta )\cdot\Phi
_{n}^{X^{^{\prime }}}((b_{k})_{k=1}^{n}).  \label{add3}
\end{equation}%
Further, there exist $u_{k}\in X$, $u_k\ge 0$, $k=1,\dots ,n$, which satisfy the
conditions: 
$\Vert u_{k}\Vert _{X}=1$ and%
\begin{equation*}
1-\delta <\left\langle u_{k},v_{k}\right\rangle =\int\nolimits_{T} u_{k}( t) v_{k}( t)\,d\mu \leq 1,1\leq k\leq n.
\end{equation*}%
Let us take for $w_{k}$ the restriction of $u_{k}$ to the support of $v_{k}$. Then,  $\left\Vert w_{k}\right\Vert _{X}\leq 1$ and 
$$\left\langle w_{k},v_{k}\right\rangle =\left\langle u_{k},v_{k}\right\rangle>1-\delta,\;1\leq k\leq n.$$ 
Therefore, since $\left\Vert w_{k}\right\Vert \leq 1,$ 
by $\left( \ref{add3}\right) $, we have
\begin{eqnarray*}
\sum_{k=1}^{n}a_{k}b_{k} &\leq &\frac{1}{1-\delta }\left\langle
\sum_{k=1}^{n}a_{k}w_{k},\sum_{k=1}^{n}b_{k}v_{k}\right\rangle\\ &\leq& \frac{1}{%
1-\delta }\max_{1\leq k\leq n}\left\Vert w_{k}\right\Vert _{X}\left\Vert
\sum_{k=1}^{n}a_{k}\frac{w_{k}}{\left\Vert w_{k}\right\Vert }\right\Vert
_{X}\left\Vert \sum_{k=1}^{n}b_{k}v_{k}\right\Vert _{X^{^{\prime }}} \\
&\leq &\frac{1+\delta }{1-\delta }\left\Vert \left( a_{k}\right)
_{k=1}^{n}\right\Vert _{X_{U}\left( n\right) }\Phi _{n}^{X^{^{\prime
}}}\left( \left( b_{k}\right) _{k=1}^{n}\right)\\&<&(1+\varepsilon) \left\Vert \left(
a_{k}\right) _{k=1}^{n}\right\Vert _{X_{U}\left( n\right) }\Phi
_{n}^{X^{^{\prime }}}\left( \left( b_{k}\right) _{k=1}^{n}\right) .
\end{eqnarray*} }
Assume now that $(b_{k})_{k=1}^{n}=\sum_{i\in A}d^{i}$, where $%
d^{i}=(d_{k}^{i})_{k=1}^{n}$ and $|A|<\infty $. Then, as above, choosing functions $v_k$ and $u_k$ in an appropriate way, we get
$$
\sum_{k=1}^{n}a_{k}d_{k}^{i}< (1+\varepsilon) \left\Vert \left(
a_{k}\right) _{k=1}^{n}\right\Vert _{X_{U}\left( n\right) }\Phi
_{n}^{X^{^{\prime }}}\left(d^i\right),\;\;i\in A,$$
and hence 
\begin{equation*}
\sum_{k=1}^{n}a_{k}b_{k}=\sum_{i\in A}\sum_{k=1}^{n}a_{k}d_{k}^{i}<
(1+\varepsilon )\sum_{i\in A}\Phi _{n}^{X^{^{\prime }}}(d^{i})\Vert
(a_{k})_{k=1}^{n}\Vert _{X_{U}(n)}.
\end{equation*}%
Passing to the infimum over above representations of $(b_{k})_{k=1}^{n}$, we
conclude that for all $n\in \mathbb{N}$ 
\begin{equation*}
\sum_{k=1}^{n}a_{k}b_{k}< (1+\varepsilon )\Vert (b_{k})_{k=1}^{n}\Vert
_{(X^{^{\prime }})_{L}(n)}\Vert (a_{k})_{k=1}^{n}\Vert _{X_{U}(n)}\leq
(1+\varepsilon )\Vert b\Vert _{(X^{^{\prime }})_{L}}\Vert a\Vert _{X_{U}},
\end{equation*}%
which implies $\left( \ref{add3a}\right) $. Thus, since $\varepsilon>0$ is arbitrary, the embedding $(X')_{L}\overset{1}{\hookrightarrow }(X_{U})^{\prime }$ is proved. 

To prove the second embedding, let us apply the first one to the K\"{o}the dual space $X^{^{\prime }}.$ It follows that 
\begin{equation}
(X'')_{L}\overset{1}{\hookrightarrow }( ( X^{^{\prime }})_{U})^{^{\prime }}.
\label{dual emb}
\end{equation}
Being order semi-continuous, $X$ is isometrically embedded as a sublattice into $X^{^{\prime \prime }}$ (see Section \ref{Function lattices}). Therefore, $\mathfrak{B}_{n}\left( X\right)\subset \mathfrak{B}_{n}\left( X''\right)$ and 
so $X_L\overset{1}{\hookrightarrow }(X'')_L$. Combining this embedding together with 
\eqref{dual emb}, we obtain that $X_{L}\overset{1}{\hookrightarrow }\left( \left( X^{^{\prime }}\right) _{U}\right) ^{^{\prime }}.$ Passing to the K\"{o}the duals and using the fact that $(X')_{U}$ has the Fatou property (see Proposition \ref{Prop-XL-XU-Fatou}), we deduce
that%
\begin{equation*}
(X')_{U}=\left( (X')_{U}\right)''\overset{1}{\hookrightarrow }\left( X_{L}\right)'.
\end{equation*}%
This completes the proof. 
\end{proof}


{\vskip0.3cm }

\subsection{\label{subsec-estimates}{Optimal sequence spaces and
order estimates of Banach lattices}}

{As it follows from results obtained in \cite{AN1}, the properties of optimal sequence spaces $X_{U}$ and $X_{L}$ are largely determined by the fact which upper/lower estimates are fulfilled in $X$ and, in particular, by the
Grobler-Dodds indices of $X$. 
}

\begin{proposition}
{\cite[Proposition 3.9 and Corollary 3.11]{AN1} 
\label{coincidence with lp} 
For every Banach lattice $X$ we have: }

(i) $X_{U}\overset{1}{\hookrightarrow }\ell_{\delta(X)}$ and $%
{\ell_p}\hookrightarrow X_{U}$ for every $p<\delta(X)$; 

{(ii) $\ell_{\sigma (X)}\overset{1}{\hookrightarrow }X_{L}$ and $%
X_{L}\hookrightarrow \ell_{q}$ for every $q>\sigma (X)$; }

{(iii) $X_U={\ell_p}$ if and only if $p=\delta(X)$ and $X$ admits an upper 
$\delta(X)$-estimate; }

{(iv) $X_{L}=\ell_{q}$ if and only if $q=\sigma (X)$ and $X$ admits a
lower $\sigma (X)$-estimate; }

{(v) $\delta(X_U)=\delta(X)$ and $\sigma(X_L)=\sigma(X)$. }
\end{proposition}

\begin{example}
{\label{Lorentz} 
Let $1<p<\infty $, $1\leq q\le \infty $ and let $%
{\ell_{p,q}}$ be the Lorentz sequence spaces (see Section \ref{RI}).
It is well known that $\delta({\ell_{p,q}})=\min (p,q)$, $\sigma ({\ell_{p,q}})=\max (p,q)$%
, and moreover, that ${\ell_{p,q}}$ admits an upper $\delta({\ell_{p,q}})$-estimate and a
lower $\sigma ({\ell_{p,q}})$-estimate (see e.g. \cite[Theorem~3]{D01}).
Consequently, by Proposition \ref{coincidence with lp}, $({\ell_{p,q}})_{U}=\ell_{%
\min (p,q)}$ and $({\ell_{p,q}})_{L}=\ell_{\max (p,q)}$. }
\end{example}

{Here, we prove a compliment to Proposition \ref{coincidence with lp}, which characterizes Banach lattices $X$ that satisfy an {equal-norm}
upper (resp. lower) $p$-estimate in terms of embeddings of the optimal upper
space $X_{U}$ (resp. optimal lower space $X_{L}$). Observe that equal-norm
estimates were used in \cite{CwNiSc03} for the classification of
decomposable pairs of Banach function spaces. }

\begin{proposition}
{\label{Tsir4} Let $1<p<\infty $ and let $X$ be a Banach lattice. Then, we have: }

{(a) $X$ satisfies an equal-norm upper $p$-estimate with constant $C$
if and only if ${\ell_{p,1}}\overset{C}{\hookrightarrow } X_U$;}

(b) if $X_L\overset{C}{\hookrightarrow } {\ell_{p,\infty}}$, then $X$ satisfies an equal-norm lower $p$-estimate with constant $C$. 

Assume that $X$ is an order semi-continuous Banach function space. Then, $X$ satisfies an equal-norm lower $p$-estimate with constant $C$ if and only if $X_L\overset{C}{\hookrightarrow } {\ell_{p,\infty}}$.
\end{proposition}

\begin{proof}
{(a) Let ${\ell_{p,1}}\overset{C}{\hookrightarrow }X_{U}$. By the
definition of $X_{U}$ and the fact that the fundamental function $\phi_{{\ell_{p,1}}}$ of ${\ell_{p,1}}$ is equal to $n^{1/p}$, for every $n\in \mathbb{N}$ and $\left\{ u_{k}\right\} _{k=1}^{n}\in \mathfrak{B}%
_{n}\left( X\right) $, we have 
\begin{equation*}
\Big\|\sum_{k=1}^{n}u_{k}\Big\|_{X}\leq \Big\|\sum_{k=1}^{n}e_{k}\Big\|%
_{X_{U}}\leq C\Big\|\sum_{k=1}^{n}e_{k}\Big\|_{{\ell_{p,1}}}=Cn^{1/p},
\end{equation*}%
which implies that $X$ admits an equal-norm upper $p$-estimate with constant 
$C$ (see Section \ref{Ban latt}).}

{Conversely, assume that $X$ satisfies an equal-norm upper $p$%
-estimate with constant $C$. Then, for every $n\in\mathbb{N}$ and any $\varepsilon>0$
there is $\left\{ u_{k}\right\}_{k=1}^{n}\in\mathfrak{B}_{n}\left( X\right) $
such that 
\begin{equation*}
\phi_{X_U}(n)=\Big\|\sum_{k=1}^n e_k\Big\|_{X_U}\le (1+\varepsilon)\Big\|%
\sum_{k=1}^n u_k\Big\|_X\le C(1+\varepsilon)n^{1/p},\;\;n\in\mathbb{N}.
\end{equation*}
Hence, as was observed in Section \ref{RI}, it follows the embedding ${\ell_{p,1}}%
\overset{C}{\hookrightarrow } X_U$. }

{(b) By the definition of the $X_L$-norm, for any $n\in\mathbb{N}$ and $\left\{ u_{k}\right\}_{k=1}^{n}\in\mathfrak{B}_{n}\left( X\right) $,  we have 
\begin{equation*}
n^{1/p}=\Big\|\sum_{k=1}^n e_k\Big\|_{p,\infty}\le C\Big\|\sum_{k=1}^n e_k%
\Big\|_{X_L}\le C\Phi_n\Big(\sum_{k=1}^n e_k\Big)\le C\Big\|\sum_{k=1}^n u_k%
\Big\|_X.
\end{equation*}
Therefore, $X$ satisfies an equal-norm lower $p$-estimate with constant $C$. 
}

To complete the proof, it remains to show that the converse to assertion (b) holds  whenever $X$ is an order semi-continuous Banach function space.

{Let $X$ admit an equal-norm lower $p$-estimate. Since $X^{\prime }$ satisfies then an equal-norm upper $q$-estimate, $1/p+1/q=1$, with the same constant
(see, for instance, \cite[Lemma 2.1]{CwNiSc03}), by the part (a) of this proposition, we obtain that $\ell_{q,1}\overset{C}{\hookrightarrow }(X')_{U}$.
Hence, by duality (see Section \ref{RI}), we get 
\begin{equation*}
((X')_{U})^{\prime }\overset{C}{\hookrightarrow }(\ell_{q,1})^{\prime
}={\ell_{p,\infty}}.
\end{equation*}%
On the other hand, since $X$ is order semi-continuous, it is isometrically embedded as a sublattice into $X^{^{\prime \prime }}$. Therefore, $\mathfrak{B}_{n}\left( X\right)\subset \mathfrak{B}_{n}\left( X''\right)$, which implies that 
$X_{L}\overset{1}{\hookrightarrow }(X^{\prime \prime })_{L}$. Moreover, the dual space $X'$ is order semi-continuous, and so from Proposition \ref{Prop-Dual.I} it follows that $(X^{\prime \prime })_{L}\overset{1}{\hookrightarrow }((X')_{U})^{\prime }$.
Combining the last embeddings, we conclude that $X_{L}\overset{C}{%
\hookrightarrow }{\ell_{p,\infty}}$, and the proof is completed. }
\end{proof}

\subsection{\label{subsec-interpol-related}A
characterization of $L_p(\mu)$-spaces in terms of optimal sequence spaces.}
\label{Equal-norm}

One can easily see that if $X=L_{p}(\mu )$ for some $\sigma $%
-finite measure space $(T,\Sigma ,\mu )$ and $1\leq p<\infty $, then $%
X_{U}=X_{L}(={\ell_p})$. We show here that, under some conditions, the converse holds as well (for a similar result see \cite{CwNiSc03}).

\begin{theorem}
\label{Th-equal-optimal-space-class}
Let $X$ be an order continuous Banach lattice with the Fatou property such that $X_{U}=X_{L}$ with equivalence of norms.
Then $X$ is order isomorphic either to $L_{p}(\mu)$-space, for some $%
1\leq p<\infty $ and measure space $(T,\Sigma,\mu)$, or to $c_0(\Gamma)$, for some set $\Gamma$.
\end{theorem}
\begin{proof}
Let $\{u_{n}\}_{n=1}^{\infty }$ be a sequence of pairwise disjoint elements  from $X$. Then, from the definition of optimal upper and lower spaces it follows that for every $n\in \mathbb{N}$
$$
\Big\|\sum_{k=1}^n \|u_k\|_X e_{k}\Big\|_{X_L}\le \Big\|\sum_{k=1}^n u_k\Big\|_X\le \Big\|\sum_{k=1}^n \|u_k\|_X e_{k}\Big\|_{X_U}.
$$
Since by the assumption $X_U=X_L$, there exists a constant $C>0$ such that for all $n\in\mathbb{N}$ and  $(a_k)_{k=1}^n\in X_L$ it holds
$$
\Big\|\sum_{k=1}^n a_ke_{k}\Big\|_{X_U}\le C\Big\|\sum_{k=1}^n a_ke_{k}\Big\|_{X_L}.
$$
Combining this inequality with the preceding one and using the Fatou property, we conclude that for each sequence $\{u_{k}\}_{k=1}^{\infty}$ of  pairwise disjoint elements from $X$ we get 
$$
\Big\|\sum_{k=1}^\infty \|u_k\|_X e_{k}\Big\|_{X_L}\le \Big\|\sum_{k=1}^\infty u_k\Big\|_X\le C\Big\|\sum_{k=1}^\infty \|u_k\|_X e_{k}\Big\|_{X_L}
$$
This inequality means that the space $X$ is {\it decomposable} with
constant $C$\footnote{This notion was introduced by Cwikel in \cite{Cwi84} in connection with the interpolation theory of operators; see also \cite{AN1}.}, i.e., the assumption that sequences $\{x_{n}\}_{n=1}^{\infty
}$ and $\{y_{n}\}_{n=1}^{\infty }$ of pairwise disjoint elements from $X$ satisfy the conditions: $\sum_{n=1}^{\infty }x_{n}\in X$
and $\left\Vert y_{n}\right\Vert _{X}\leq \left\Vert x_{n}\right\Vert _{X}$, 
$n\in N$, implies that $\sum_{n=1}^{\infty }y_{n}\in X$ and 
\begin{equation*}  \label{reldec}
\left\Vert \sum_{n=1}^{\infty }y_{n}\right\Vert _{X}\leq C\left\Vert
\sum_{n=1}^{\infty }x_{n}\right\Vert _{X}.
\end{equation*}%
Since $X$ is order continuous, applying now \cite[Theorem 1.b.12]{LT79}, we complete the proof.

\end{proof}

\subsection{\label{subsec-interpol-related}A characterization of optimal sequence spaces.}
\label{Equal-norm}

Let $E$ be an order semi-continuous Banach sequence lattice. According to \cite[Proposition 3.8(a)]{AN1}, $E\overset{1}{\hookrightarrow }E_{L}$ and $E_{U}\overset{1}{\hookrightarrow }E$.
In this section, we will be interested in identification of Banach sequence lattices $E$ such that the opposite embeddings hold. In other words, we intend to find conditions, under which $E=X_{L}$ or $E=X_{U}$ for some Banach lattice $X$. Observe that, in view of Theorem \ref{Th_XL_XU_Prop} and Proposition \ref{Prop-XL-XU-Fatou}, this is possible only if $E$ is order isomorphic to a symmetric sequence space with the Fatou property. 

\begin{proposition}
\label{Tsir1a} Let $E$ be a Banach sequence lattice satisfying the following property: there exists a constant $C>0$ such that for every  $n\in \mathbb{N}$, $a_k\in \mathbb{R}$, $k=1,\dots,n$, $\left\{ u_{k}\right\}_{k=1}^{n}\in\mathfrak{B}_{n}\left( E\right) $ there is an injective mapping $\pi:\,\{1,\dots,n\}\to \mathbb{N}$ such that  
\begin{equation}
\label{add1}
\Big\|\sum_{k=1}^n a_ke_{\pi(k)}\Big\|_E\le C\Big\|\sum_{k=1}^n a_ku_k\Big\|_E.
\end{equation}

Then, $E_L$ is the minimal symmetric space with the Fatou property that contains $E$. More precisely, if $F$ is a symmetric space with the Fatou property such that $E\overset{C_1}{\hookrightarrow} F$, then the embedding $E_L\overset{CC_1}{\hookrightarrow} F$ holds.
\end{proposition}

\begin{proof}
 Let $n\in \mathbb{N}$ and $a=(a_k)_{k=1}^n$ be arbitrary. From \eqref{add1} and the definition of the functional $\Phi_n(\cdot)$ it follows that 
$$
\Phi_n(a)\ge \frac1C\inf_{\pi}\Big\|\sum_{k=1}^n a_ke_{\pi(k)}\Big\|_E,$$
where the infimum is taken over all injective mappings $\pi:\,\{1,\dots,n\}\to \mathbb{N}$. 
Therefore, assuming that $F$ is a symmetric space satisfying $E\overset{C_1}{\hookrightarrow} F$, for every $\varepsilon>0$ and some representation $a=\sum_{i\in A}a^i$, where $a^i=(a^i_k)$ and $|A|<\infty$, we obtain
\begin{eqnarray*}
\|a\|_{E_L}&\ge&  (1-\varepsilon)\sum_{i\in A}\Phi_n(a^i)\ge \frac{1-\varepsilon}{C} \sum_{i\in A}\inf_{\pi}\Big\|\sum_{k=1}^n a_k^ie_{\pi(k)}\Big\|_E\\&\ge& \frac{(1-\varepsilon)}{CC_1} \sum_{i\in A}\inf_{\pi}\Big\|\sum_{k=1}^n a_k^ie_{\pi(k)}\Big\|_F\\&=&\frac{(1-\varepsilon)}{CC_1} \sum_{i\in A}\Big\|\sum_{k=1}^n a_k^ie_{k}\Big\|_F\\&\ge&\frac{(1-\varepsilon)}{CC_1} \Big\|\sum_{k=1}^n\sum_{i\in A} a_k^ie_{k}\Big\|_F\\&=&\frac{(1-\varepsilon)}{CC_1}\|a\|_{F}.
\end{eqnarray*}
Since $F$ has the Fatou property and $\varepsilon>0$ is arbitrary, this implies that $E_L\overset{CC_1}{\hookrightarrow} F$. Hence, the proof is completed.
\end{proof}

The proof of the next result follows by the same lines; so we skip it. 


\begin{proposition}
\label{Tsir1aX_U} Let a Banach sequence lattice $E$ satisfy the following property: there exists a constant $C>0$ such that for every  $n\in \mathbb{N}$, $a_k\in \mathbb{R}$, $k=1,\dots,n$, $\left\{ u_{k}\right\}_{k=1}^{n}\in\mathfrak{B}_{n}\left( E\right) $ there is an injective mapping $\pi:\,\{1,\dots,n\}\to \mathbb{N}$ such that  
\begin{equation}
\label{add2a}
\Big\|\sum_{k=1}^n a_ku_k\Big\|_E\le C\Big\|\sum_{k=1}^n a_ke_{\pi(k)}\Big\|_E.
\end{equation}

Then, $E_U$ is the maximal symmetric space that is contained into $E$. More precisely, if $F$ is a symmetric space such that $F\overset{C_1}{\hookrightarrow} E$, then the embedding $F\overset{CC_1}{\hookrightarrow} E_U$ holds.
\end{proposition}

Let $X$ be a Banach lattice. As was proved in \cite{AN1}, $X_U$ (resp. $X_L$) is an upper (resp. lower) semi-homogeneous space (see Section \ref{Lower/upper semi-homogeneous}). More precisely, since the spaces $X_U$ and $X_L$ are symmetric, we have the following.


\begin{proposition}\cite[Proposition 3.7]{AN1}
\label{Prop_main} 
For every Banach lattice $X$ we have the following:

$\left( i\right)$. For any $m\in\mathbb{N}$ and every pairwise disjoint elements $u_{k}\in X_U$, $k=1,2,\dots,m$, we have 
\begin{equation*}
\left\Vert \sum_{k=1}^{m}u_{k}\right\Vert _{X_{U}}\leq \left\Vert
\sum_{k=1}^{m}\left\Vert u_{k}\right\Vert _{X_{U}}e_{k}\right\Vert _{X_{U}};
\end{equation*}

$\left( ii\right)$. For any $m\in\mathbb{N}$ and every pairwise disjoint elements $u_{k}\in X_L$, $k=1,2,\dots,m$, we have
\begin{equation*}
\left\Vert \sum_{k=1}^{m}\left\Vert u_{k}\right\Vert
_{X_{L}}e_{k}\right\Vert _{X_{L}}\leq \left\Vert
\sum_{k=1}^{m}u_{k}\right\Vert _{X_{L}}.
\end{equation*}
\end{proposition}

Applying Propositions \ref{Prop_main}, \ref{Tsir1a} and \ref{Tsir1aX_U}, we get the following characterization of symmetric sequence spaces which can be represented as the lower (resp. upper) optimal space of some Banach lattice.

\begin{theorem}
 \label{add cor1} 
 Let $E$ be a symmetric sequence space with the Fatou property. Then, we have:
 
(a) $E=X_L$ for some Banach lattice $X$ (equivalently, $E=E_L$) if and only if $E$ is lower semi-homogeneous, i.e., there is a constant $C>0$ such that the inequality 
\begin{equation}
\label{add1new}
\Big\|\sum_{k=1}^n a_ke_{k}\Big\|_E\le C\Big\|\sum_{k=1}^n a_ku_k\Big\|_E
\end{equation}
holds for every  $n\in \mathbb{N}$, $a_k\in \mathbb{R}$, $k=1,\dots,n$, and $\left\{ u_{k}\right\}_{k=1}^{n}\in\mathfrak{B}_{n}\left( E\right)$;

(b) $E=X_U$ for some Banach lattice $X$ (equivalently, $E=E_U$) if and only if $E$ is upper semi-homogeneous, i.e.,there is a constant $C>0$ such that the inequality 
\begin{equation}
\label{add2anew}
\Big\|\sum_{k=1}^n a_ku_k\Big\|_E\le C\Big\|\sum_{k=1}^n a_ke_{k}\Big\|_E
\end{equation}
holds for every  $n\in \mathbb{N}$, $a_k\in \mathbb{R}$, $k=1,\dots,n$, and $\left\{ u_{k}\right\}_{k=1}^{n}\in\mathfrak{B}_{n}\left( E\right)$.
\end{theorem}

Combining the latter result with Proposition \ref{coincidence with lp} (see also \cite[Proposition 3.9(i),(ii)]{AN1}), we get the following embeddings for Banach sequence lattices which satisfy either \eqref{add1new} or \eqref{add2anew}. 

\begin{corollary}
\label{Tsir6}

(a) Assume that $E$ be a symmetric sequence space with the Fatou property such that for every  $n\in \mathbb{N}$, $a_k\in \mathbb{R}$, $k=1,\dots,n$, and $\left\{ u_{k}\right\}_{k=1}^{n}\in\mathfrak{B}_{n}\left( E\right)$ inequality \eqref{add1new} holds. Then, 
$l_{\sigma(E)}\overset{1}{\hookrightarrow }E$ and $E{\hookrightarrow } {\ell_q}$ for each $q>\sigma(E)$.

(b) Assume that $E$ be a symmetric sequence space such that for every  $n\in \mathbb{N}$, $a_k\in \mathbb{R}$, $k=1,\dots,n$, and $\left\{ u_{k}\right\}_{k=1}^{n}\in\mathfrak{B}_{n}\left( E\right)$ inequality \eqref{add2anew} holds. Then,
$E\overset{1}{\hookrightarrow }l_{\delta(E)}$ and ${\ell_q}{\hookrightarrow } E$ for each $q<\delta(E)$.

\end{corollary}




\subsection{Tensor product in semi-homogeneous symmetric spaces}

Let $a=(a_k)_{k=1}^\infty$ and $b=(b_k)_{k=1}^\infty$ be two bounded equences of real numbers. By $a\otimes b$ we will denote the {\it tensor product} of $a$ and $b$, i.e., the sequence $(a_ib_j)_{i,j=1}^\infty$. 

Let $E,F$ and $G$ be symmetric sequence spaces. The tensor product operator defined by
\begin{equation}
\label{tensor}
(a,b)\mapsto a\otimes b
\end{equation}
is said to be bounded from direct product $E\times F$ in $G$ if there is a constant $C>0$ such that for all $a\in E$ and $b\in F$ it holds:
$$
\|(a\otimes b)^*\|_G\le C\|a\|_E\|b\|_F,$$
where $(a\otimes b)^*$ is the permutation of the sequence $(|a_ib_j|)_{i,j=1}^\infty$ in decreasing order.


Following \cite{Jun02}, we represent the tensor product for finitely supported sequences slightly in a different way. Let 
$a=(a_k)$ and $b=(b_k)$ be two finitely supported sequences, say, $a_k=b_k=0$ if $k>n$. Then the tensor product of $a$ and $b$ may be defined as follows:
$$
a\otimes b:= \sum_{i=1}^n a_i\tilde{b}^i,$$
where $\tilde{b}^i$ are the following "shifts" of the sequence $b$:
$$
\tilde{b}^i:=\sum_{k=1}^n b_ke_{(i-1)n+k}$$
(in particular, $\tilde{b}^1=b$).



Let us show that operator \eqref{tensor} has some special properties in semi-homogeneous symmetric spaces.

Recall that $\phi_{E}$ is the fundamental function of a symmetric sequence space $E$
(see Section \ref{RI}).
 
\begin{proposition}
\label{Tsir5} 

(a) If $E$ is an upper semi-homogeneous symmetric sequence space, then operator
\eqref{tensor} is bounded from $E\times E$ into $E$.

In particular, for some $C>0$
$$
\phi_E(mn)\le C\phi_E(n)\phi_E(m),\;\;n,m\in\mathbb{N}.$$

(b) If $E$ is a lower semi-homogeneous symmetric sequence space, we have 
$$
\|a\|_E\|b\|_E\le C\|a\otimes b\|_E.$$

In particular, 
$$
\phi_E(n)\phi_E(m)\le C\phi_E(mn),\;\;n,m\in\mathbb{N}.$$
\end{proposition}

\begin{proof}
We prove only assertion (b) (the proof of (a) follows by the same lines).

By Theorem  \ref{add cor1}(a), $E=E_L$ and hence, thanks to Proposition \ref{Prop-XL-XU-Fatou}, the space $E$ has the Fatou property. Consequently, it suffices to consider
only finitely supported sequences $a$ and $b$. 

Assume that $a_k=b_k=0$ for all $k>n$. Then, taking into account that $\|\tilde{b}^i\|_E=\|b\|_E$ for all $i=1,\dots,n$, by the assumption, applied to the sequences $\tilde{b}^i$, $i=1,\dots,n$, we get  
$$
\|a\|_E\|b\|_E=\Big\|\sum_{i=1}^n a_i\|\tilde{b}^i\|_Ee_i\Big\|_E\le C\Big\|\sum_{i=1}^n a_i\tilde{b}^i\Big\|_E=C\|a\otimes b\|_E.$$
Thus, the first assertion is proved. Since the second assertion is an immediate consequence of the first one, the proof is completed.
\end{proof}

\vskip0.3cm

\section{\label{Appl}
Applications to Lorentz and Orlicz sequence spaces}

\subsection{\label{Lor}
Lorentz spaces}

Let $\{w_k\}_{k=1}^\infty$ be a nonincreasing sequence of positive numbers such that $w_1=1$, $\lim_{k\to\infty} w_k=0$ and $\sum_{k=1}^\infty w_k=\infty$, $1\le q<\infty$, and let $\lambda_q(w)$ be the Lorentz sequence space (see Section \ref{RI}). Recall also that $\mu_{\lambda_q(w)}$ and $\nu_{\lambda_q(w)}$ are the lower and upper fundamental indices of the fundamental function $\phi_{\lambda_q(w)}$ (see formulae \eqref{fund Lor}, \eqref{Boyd ind Lor} and \eqref{Boyd ind Lor1}).

\begin{proposition}
\label{Lorentz7}  We have the following:

(i) $(\lambda_q(w))_U={\ell_q}$;

(ii) $(\lambda_q(w))_L=\ell_{1/\mu_{\lambda_q(w)}}$ (equivalently, $(\lambda_q(w))_L={\ell_p}$ for some $1\le p<\infty$) if and only if 
there exists a constant $C>0$ such that for any $n,l\in\mathbb{N}$ it holds
\begin{equation}
\label{Lor: assump}
\sum_{k=1}^{n}w_k\le Cl^{-q\mu_{\lambda_q(w)}}\sum_{k=1}^{ln}w_k.
\end{equation}

In particular, in the case when $\mu_{\lambda_q(w)}=1/q$ inequality \eqref{Lor: assump} 
is fulfilled if and only if $w$ is equivalent to a constant sequence, that is, $\lambda_q(w)={\ell_q}$.
\end{proposition}
\begin{proof}
(i) Since $\lambda_q(w)$ is the $q$-convexification of the space $\lambda_1(w)$ (see e.g. \cite[1d, p.~53]{LT79}), $\lambda_q(w)$ satisfies an upper $q$-estimate. Moreover, it does not satisfy an upper $r$-estimate for any $r>q$ (see \cite{R-81}, \cite{Ca-82} or \cite[Theorem 3]{KMP-98}). Thus, $\delta(\lambda_q(w))=q$. Combining this fact with Proposition \ref{coincidence with lp}(iii), we get the desired result.

(ii) By \cite{Ca-82} (see also \cite[Theorem 10]{KMP-98}), $\lambda_q(w)$ satisfies a lower $p$-estimate if and only if 
the sequence $n^{-q/p}S_n$, where  $S_n:=\sum_{k=1}^n w_k$, is equivalent to an increasing sequence. Clearly, this holds if and only if there exists a constant $C_p>0$ such that for all positive integers $j,l$ we have
\begin{equation}
\label{Lor: dil}
S_j\le C_p l^{-q/p}S_{lj}.
\end{equation}
We claim that inequality \eqref{Lor: dil}  holds if $p>1/\mu_{\lambda_q(w)}$. Indeed, combining the last condition with formula \eqref{Boyd ind Lor}, we obtain 
$$
S_j\le C_p'2^{-nq/p}S_{2^nj}$$
for some $C_p'>0$ and all $n=0,1,\dots$ and $j\in\mathbb{N}$. Now, for every $l\in\mathbb{N}$, choosing $n$ so  that $2^n\le l < 2^{n+1}$ and taking into account that $\{w_k\}_{k=1}^\infty$ is nonincreasing, we infer  
$$
S_j\le C_p'2^{-(n+1)q/p}S_{2^{n+1}j}\le C_p'l^{-q/p}\Big(\sum_{k=1}^{lj} w_k+\sum_{k=lj+1}^{2^{n+1}j} w_k\Big)\le 2C_p'l^{-q/p}S_{lj},$$
and our claim is proved.
Moreover, from \eqref{Boyd ind Lor}  it follows that the opposite inequality $p<1/\mu_{\lambda_q(w)}$ implies that
$$
\lim_{n\to\infty}\frac{2^{nq/p}S_j}{S_{2^nj}}=\infty.$$
Hence, then \eqref{Lor: dil} fails to hold. 

As a result of the above discussion, we see that the lower Grobler-Dodds index $\sigma(\lambda_q(w))$ is equal to $1/\mu_{\lambda_q(w)}$.
Therefore, inequality \eqref{Lor: assump} holds if and only if the space $\lambda_q(w)$ satisfies a lower $\sigma(\lambda_q(w))$-estimate. Thus, applying Proposition \ref{coincidence with lp}(iv), we conclude that \eqref{Lor: assump} is equivalent to the fact that $(\lambda_q(w))_L=\ell_{1/\mu_{\lambda_q(w)}}$ (or $(\lambda_q(w))_L={\ell_p}$ for some $1\le p<\infty$).

Finally, since $\{w_k\}_{k=1}^\infty$ is a nonincreasing sequence, inequality \eqref{Lor: assump} may be valid only if $\mu_{\lambda_q(w)}\le 1/q$. Hence, assuming that $\mu_{\lambda_q(w)}=1/q$, we obtain the last assertion of the proposition.


\end{proof}


From Proposition \ref{Lorentz7}(i) and Theorem \ref{add cor1} it follows

\begin{corollary}
\bigskip \label{cor for Lor} 
The following conditions are equivalent:

(a) for all $n\in\mathbb{N}$ and pairwise disjoint sequences $u_i\in \lambda_q(w)$ it holds
$$
\Big\|\sum_{i=1}^n u_i\Big\|_{\lambda_q(w)}\le \Big\|\sum_{i=1}^n \|u_i\|_{\lambda_q(w)}e_i\Big\|_{\lambda_q(w)};$$

(b) $w$ is equivalent to a constant sequence and hence $\lambda_q(w)={\ell_q}$.
\end{corollary}

In \cite{CaLin}, Casazza and Bor-Luh Lin proved that the space $\lambda_q(w)$ is lower semi-homogeneous if and only if 
\begin{equation}
\label{Lor: did}
\sup_{n\in\mathbb{N}}\frac{\sum_{k=1}^n d_k}{\sum_{k=1}^nw_k}<\infty,
\end{equation}
where $\{d_k\}_{k=1}^\infty$ is the permutation of the sequence $\{w_iw_j\}_{i,j=1,2,\dots}$ in decreasing order. Combining this result together with Theorem \ref{add cor1}(a) we obtain the following.

\begin{corollary}
\label{cor for Lor for lower} 
Let $1\le q<\infty$, $\{w_k\}_{k=1}^\infty$ be a nonincreasing sequence of positive numbers such that $w_1=1$, $\lim_{k\to\infty} w_k=0$ and $\sum_{k=1}^\infty w_k=\infty$, $1\le q<\infty$. The following conditions are equivalent:

(a) $(\lambda_q(w))_L=\lambda_q(w)$;

(b) the sequence $\{w_k\}$ satisfies condition \eqref{Lor: did}.
\end{corollary}

\begin{example}
(i) Let $1<q<\infty$, $0<\alpha<1$, and $w_k=k^\alpha- (k-1)^\alpha$, $k=1,2,\dots$ Then, the Loremtz space $\lambda_q(w)$ is just ${\ell_{p,q}}$, where $p=q/\alpha$. In this case $\mu_{\lambda_q(w)}=\nu_{\lambda_q(w)}=\alpha/q$ and, by Proposition \ref{Lorentz7},   $(\lambda_q(w))_U={\ell_q}$ and $(\lambda_q(w))_L=\ell_{q/\alpha}$ (see Example \ref{Lorentz}). Clearly, the sequence $\{w_k\}_{k=1}^\infty$ does not satisfy condition \eqref{Lor: did} for every $0<\alpha<1$.

(b) (see \cite[Corollary 8]{CaLin2}). If $v_1=v_2=1$, $v_k=1/\log k$, $k=3,4,\dots$,  condition \eqref{Lor: did} holds. Thus, by Corollary \ref{cor for Lor for lower}, $(\lambda_q(v))_L=\lambda_q(v)$ for every $1\le q<\infty$. 

Since
$$
\sum_{k=3}^j v_k\asymp \int_2^j \frac{dt}{\log t},\;\;j\ge 3,$$ 
we have 
$$
\sum_{k=1}^j v_k\asymp \frac{j}{\log(2j)},\;\;j\ge 1.$$   
Combining the latter equivalence with formulae \eqref{Boyd ind Lor} and \eqref{Boyd ind Lor1}, one can deduce that $\mu_{\lambda_q(v)}=\nu_{\lambda_q(v)}=1/q$. In particular, this implies that the sequence $v$ fails to satisfy condition \eqref{Lor: assump}.
\end{example}

\vskip0.2cm

\subsection{\label{Orl}
Orlicz spaces}

Let $N$ be an Orlicz function on $[0,\infty)$ such that $N(0)=0$ and $N(1)=1$ and let ${\ell_N}$ be the Orlicz sequence space (see Section \ref{RI}).

We start with giving rather simple necessary and sufficient conditions under which ${\ell_N}$ is a lower (resp. an upper) semi-homogeneous space.


\begin{proposition}
\label{Tsir7} Let $N$ satisfy the $\Delta_2$-condition at zero. The following conditions are equivalent:

(a) ${\ell_N}$ is lower semi-homogeneous;

(b) ${\ell_N}=({\ell_N})_L$;

(c) ${\ell_N}=X_L$ for some Banach lattice $X$;

(d)  there exists a constant $C\ge 1$ such that for all $a,b\in {\ell_N}$ it holds
$$
\|a\|_{{\ell_N}}\|b\|_{{\ell_N}}\le C\|a\otimes b\|_{{\ell_N}};$$

(e)  there exists a constant $C_1$ such that for all $0<s,t\le 1$ we have
$$
N(s)N(t)\le C_1N(st).$$
\end{proposition}

\begin{proof}
Observe first that the implication $(a)\Rightarrow (b)$ is an immediate consequence of Theorem  \ref{add cor1}(a), while the implications $(b)\Rightarrow (c)$ and $(d)\Rightarrow (e)$ are obvious.

$(c)\Rightarrow (d)$. From (c), Theorem \ref{add cor1}(a) and Proposition \ref{Tsir5}(b) it follows that
$$
\phi_{{\ell_N}}(n)\phi_{{\ell_N}}(m)\le C'\phi_{{\ell_N}}(mn),\;\;n,m\in\mathbb{N}.$$
Now, combining this together with equality \eqref{fund Orl}, the $\Delta_2$-condition and convexity of $N$, one easily gets (d).

$(e)\Rightarrow (a)$. 
We need to show that for every  $n\in \mathbb{N}$, $\left\{ u_{k}\right\}_{k=1}^{n}\in\mathfrak{B}_{n}\left( X\right) $ and $a_k\in \mathbb{R}$, $k=1,\dots,n$ it holds
\begin{equation}
\label{add1new Orlicz}
\Big\|\sum_{k=1}^n a_ke_{k}\Big\|_{{\ell_N}}\le C_1\Big\|\sum_{k=1}^n a_ku_k\Big\|_{{\ell_N}},
\end{equation}
where $C_1$ is the constant from condition (e).
Without loss of generality, we may assume that $a_k\ge 0$, $u_k=\sum_{i\in A_k}b_ie_i$, where $A_k$ are pairwise disjoint sets of positive integers and $b_i\ge 0$ for all $i\in \cup_{k=1}^n A_k$. Moreover, by homogeneity, it can be assumed also that $a_k\le 1$, $k=1,\dots,n$. Then, since $\|u_k\|_{{\ell_N}}=1$, we have $\sum_{i\in A_k}N(b_i)=1,$ $k=1,\dots,n$ (in particular, this implies that $b_i\le 1$ because of the fact that $N$ is an increasing function and $N(1)=1$). Therefore, if
$$
\Big\|\sum_{k=1}^n a_ku_k\Big\|_{{\ell_N}}=1,$$ 
from the assumption it follows that
$$ 
1=\sum_{k=1}^n\sum_{i\in A_k}N(a_kb_i^k)\ge C_1^{-1}\sum_{k=1}^n N(a_k)\cdot\sum_{i\in A_k}N(b_i^k)=C_1^{-1}\sum_{k=1}^n N(a_k).$$
Since $C_1\ge 1$ and $N$ is an Orlicz function, this implies that 
$$
\Big\|\sum_{k=1}^n a_ke_{k}\Big\|_{{\ell_N}}\le C_1.$$
Thus, \eqref{add1new Orlicz} is proved and hence everything is done.  
\end{proof}

The proof of the next assertion is similar, and so it is skipped.

\begin{proposition}
\label{Tsir8} Let $N$ be an Orlicz function which satisfies the $\Delta_2$-condition at zero. The following conditions are equivalent:

(a) ${\ell_N}$ is upper semi-homogeneous;

(b) ${\ell_N}=({\ell_N})_U$;

(c) ${\ell_N}=X_U$ for some Banach lattice $X$;

(d) there exists a constant $C\ge 1$ such that for all $a,b\in {\ell_N}$ it holds
$$
\|a\otimes b\|_{{\ell_N}}\le C\|a\|_{{\ell_N}}\|b\|_{{\ell_N}};$$

(e)  there exists a constant $C_1\ge 1$ such that for all $0<s,t\le 1$ we have
$$
N(st)\le C_1N(s)N(t).$$
\end{proposition}





It is known (see e.g. \cite{KMP97}) that the Orlicz space ${\ell_N}$ admits an upper $p$-estimate (resp. a lower $p$-estimate) if and only if there exists a constant $K>0$ such that for all $v\ge 1$ and $0<u\le vu\le 1$ we have
$$
N(v)u^p\le K\cdot N(vu)$$
(resp.
$$
N(vu)\le K\cdot N(v)u^p\;).$$
Assume that $N$ satisfies the $\Delta_2$-condition at zero. Then, from easy calculations it follows that the latter inequalities hold if and only if  
there exists a constant $K'>0$ such that for all $0<s,t\le 1$  
$$
N(st)\le K'\cdot N(s)t^p$$
(resp.
$$
N(s)t^p\le K'\cdot N(st)\;).$$
Hence, taking into account formulae \eqref{fund ind Orl}, we conclude that $\delta({\ell_N})=1/\nu_{{\ell_N}}$ and $\sigma({\ell_N})=1/\mu_{{\ell_N}}$. Consequently, applying Proposition \ref{coincidence with lp}(iii),(iv), we get the following result.

\begin{proposition}
\label{TsirOrl} 
Let $N$ be an Orlicz function satisfying the $\Delta_2$-condition at zero. The following conditions (a), (b) and (c) (resp. (a'), (b') and (c')) are equivalent:

(a) $({\ell_N})_U={\ell_p}$ for some $1\le p<\infty$;

(b) $({\ell_N})_U=\ell_{1/\nu_{{\ell_N}}}$;

(c) there exists a constant $K'>0$ such that for all $0<s,t\le 1$ we have
$$
N(st)\le K'\cdot N(s)t^{1/\nu_{{\ell_N}}};$$

(a') $({\ell_N})_L={\ell_p}$ for some $1\le p<\infty$;

(b') $(\ell_N)_L=\ell_{1/\mu_{\ell_N}}$;

(c') there exists a constant $K'>0$ such that for all $0<s,t\le 1$ we have
$$
N(s)t^{1/\mu_{{\ell_N}}}\le K'\cdot N(st).$$
 
\end{proposition}

\begin{example}
Let $1<p<\infty$, $a\in\mathbb{R}$ and $N$ be an Otlicz function on $[0,\infty)$ such that $N(t)\asymp t^p\log^{a}(e/t)$ for $0<t\le 1$. 

One can easily check that $\mu_{{\ell_N}}=\nu_{{\ell_N}}=1/p$. Then, if $a\ge 0$, from Propositions  \ref{Tsir8} and  \ref{TsirOrl}, respectively, it follows that  conditions (e) and (c')  are fulfilled. Hence, $({\ell_N})_U={\ell_N}$ and $({\ell_N})_L={\ell_p}$. Similarly, in the case when $a\le 0$, conditions (e) and (c) from Propositions  \ref{Tsir7} and  \ref{TsirOrl}, respectively, hold. Therefore,  we have: $({\ell_N})_L={\ell_N}$ and $({\ell_N})_U=\ell_{p}$.

Finally, let $N(t)\asymp t\log^{a}(e/t)$, $0<t\le 1$, where $a\le 0$. Then, in the same way, applying Propositions  \ref{Tsir7} and  \ref{TsirOrl}, we get $({\ell_N})_L={\ell_N}$ and $({\ell_N})_U=\ell_{1}$.
\end{example}

\end{document}